\documentclass[11pt,reqno]{amsart}

\usepackage{hyperref}
\usepackage{graphicx}
\usepackage{color}
\usepackage{amsfonts,amssymb}
\usepackage{bbm} 

\usepackage{mathrsfs}

\usepackage{pdfsync}

\usepackage{a4wide}

\newtheorem{neu}{}
\newtheorem{Cor}[neu]{Corollary}
\newtheorem*{Cor*}{Corollary}
\newtheorem{Thm}[neu]{Theorem}
\newtheorem*{Thm*}{Theorem}
\newtheorem*{Observation*}{Observation}
\newtheorem{Prop}[neu]{Proposition}
\newtheorem*{Prop*}{Proposition}
\newtheorem{Lemma}[neu]{Lemma}
\theoremstyle{definition}\newtheorem*{Rmk*}{Remark}
\newtheorem{Rmk}[neu]{Remark}

\newtheorem*{Ex*}{Example}

\newtheorem*{Qu*}{Question}

\newcommand{\N}{\mathbb{N}}

\newcommand{\Z}{\mathbb{Z}}
\newcommand{\R}{\mathbb{R}}

\newcommand{\om}{\omega}
\newcommand{\Om}{\Omega}

\newcommand{\Ham}{\mathrm{Ham}}

\newcommand{\Crit}{\mathrm{Crit}}

\newcommand{\beq}{\begin{equation}}
\newcommand{\beqn}{\begin{equation}\nonumber}
\newcommand{\eeq}{\end{equation}}

\newcommand{\bea}{\begin{equation}\begin{aligned}}
\newcommand{\bean}{\begin{equation}\begin{aligned}\nonumber}
\newcommand{\eea}{\end{aligned}\end{equation}}

\definecolor{Urs}{rgb}{0,.7,0}
\definecolor{Peter}{rgb}{0,0,1}
\definecolor{red}{rgb}{1,0,0}

\begin{document}
\title{Square roots of Hamiltonian Diffeomorphisms}
\author{Peter Albers}
\author{Urs Frauenfelder}
\address{
    Peter Albers\\
    Mathematisches Institut\\
    Westf\"alische Wilhelms-Universit\"at M\"unster}
\email{peter.albers@wwu.de}
\address{
    Urs Frauenfelder\\
    Department of Mathematics and Research Institute of Mathematics\\
    Seoul National University}
\email{frauenf@snu.ac.kr}
\begin{abstract}
In this article we prove that on any closed symplectic manifold there exists an arbitrarily $C^\infty$-small Hamiltonian diffeomorphism not admitting a square root.
\end{abstract}
\maketitle


\section{Introduction}
Let $(M,\om)$ be a closed symplectic manifold, i.e.~$\om\in\Om^2(M)$ is a non-degenerate, closed 2-form. To a function $L:S^1\times M\to\R$ we associate the Hamiltonian vector field $X_L$ by setting
\beq
\om(X_{L_t},\cdot)=-dL_t(\cdot)
\eeq
where $L_t(x):=L(t,x)$. The flow $\phi_L^t:M\to M$ of the vector field $X_{L_t}$ is called a Hamiltonian flow. For simplicity we abbreviate
\beq
\phi_L=\phi_L^1\;.
\eeq

The Hamiltonian diffeomorphisms form the Lie group $\Ham(M,\om)$ with Lie algebra being the smooth functions modulo constants. We refer the reader to the book \cite{McDuff_Salamon_introduction_symplectic_topology} for the basics in symplectic geometry.

In this article we prove the following Theorem.
\begin{Thm}\label{thm:main}
In any $C^\infty$-neighborhood of the identity in $\Ham(M,\om)$ there exists a Hamiltonian diffeomorphism $\phi$ which has no square root, i.e.~for all Hamiltonian diffeomorphism $\psi$ (not necessarily close to the identity) 
\beq
\psi^2\neq\phi
\eeq
holds.
\end{Thm}
An immediate corollary of Theorem \ref{thm:main} is the following.
\begin{Cor}
The exponential map
\bea
\mathrm{Exp}:C^\infty(M,\R)/\R&\to\Ham(M,\om)\\
[L]&\mapsto\phi_L
\eea 
is not a local diffeomorphism.
\end{Cor}

In the proof of the Theorem we use the following beautiful observation by Milnor \cite[Warning 1.6]{Milnor_Remarks_on_infinite_dimensional_Lie_groups}. Milnor observed that an obstruction to the existence of a square root is an odd number of $2k$-cycles, see next section for details. The main work in this article is to construct an example in the symplectic category.

\subsection*{Acknowledgements}
The authors are indebted to Kaoru Ono for invaluable discussions and drawing our attention to Milnor's article \cite{Milnor_Remarks_on_infinite_dimensional_Lie_groups}.

This material is supported by the SFB 878 -- Groups, Geometry and Actions (PA) and by the Basic Research fund 20100007669 funded by the Korean government basic (UF).

\section{Milnor's observation}

We define
\beq
CM^k:=M^k\big/(\Z/k)
\eeq
where $\Z/k$ acts by cyclic shifts on $M^k$. We write elements of $CM^k$ as
\beq
[x_1,\ldots,x_{k}]\in CM^k\;.
\eeq
The space of $k$-cycles of a diffeomorphism $\phi:M\to M$ is
\beq
\mathscr{C}^{k}(\phi):=\big\{[x_1,\ldots,x_{k}]\in CM^{k}\mid \phi^j(x_i)\neq x_i\,\forall j=1,\ldots, k-1,\;\phi(x_i)=x_{i+1}\big\}\;.
\eeq
We point out that if $[x_1,\ldots,x_{k}]\in\mathscr{C}^{k}(\phi)$ then $\phi^k(x_i)=x_i$ for $i=1,\ldots,k$.

\begin{Prop}[Milnor \cite{Milnor_Remarks_on_infinite_dimensional_Lie_groups}]\label{prop:Milnor}
If $\phi=\psi^2$ then $\mathscr{C}^{2k}(\phi)$ admits a free $\Z/2$-action. In particular, $\#\mathscr{C}^{2k}(\phi)$ is even if $\mathscr{C}^{2k}(\phi)$ is a finite set.
\end{Prop}

For the convenience of the reader we include a proof of Milnor's ingenious observation.

\begin{proof}
We define 
\bea
I:\mathscr{C}^{2k}(\phi)&\to\mathscr{C}^{2k}(\phi)\\
[x_1,\ldots,x_{2k}]&\mapsto[\psi(x_1),\ldots,\psi(x_{2k})]\;.
\eea 
Since $\psi\circ\phi=\phi\circ\psi$ and $\psi^2=\phi$ the map $I$ is well-defined and an involution. We assume by contradiction that $[x_1,\ldots,x_{2k}]$ is a fixed point of $I$, i.e.~there exists $0\leq r\leq2k-1$
\beq
\psi(x_i)=x_{i+r}
\eeq
where we read indices $\Z/2k$-cyclically. Using $x_{i+r}=\phi^r(x_i)$ we get
\beq
\psi(x_i)=\phi^r(x_i)=\psi^{2r}(x_i)
\eeq
and thus
\beq
\psi^{2r-1}(x_i)=x_i\;.
\eeq
In particular,
\beq
x_i=\psi^{2r-1}(x_i)=\psi^{2r-1}(\psi^{2r-1}(x_i))=\psi^{4r-2}(x_i)=\phi^{2r-1}(x_i)\;.
\eeq
In summary we have
\beq
x_i=\phi^{2r-1}(x_i)\quad\text{and}\quad x_i=\phi^{2k}(x_i)\;.
\eeq
In general, if 
\beq
z=\phi^a(z)\quad\text{and}\quad z=\phi^b(z)
\eeq
for $a,b\in\Z$ then 
\beq
z=\phi^{\text{lcd}(a,b)}(z)
\eeq
since by the Euclidean algorithm there exists $n_1,n_2\in\Z$ with
\beq
\text{lcd}(a,b)=n_1a+n_2b\;.
\eeq
In our specific situation $2r-1$ is odd and $2k$ is even and thus 
\beq
1\leq \text{lcd}(2r-1,2k)<2k
\eeq
contradicting the assumption $ \phi^j(x_i)\neq x_i\,\forall j=1,\ldots, 2k-1$. This proves the Proposition.
\end{proof}

\section{Proof of Theorem \ref{thm:main}}

Let $(M,\om)$ be a closed symplectic manifold. We fix a Darboux chart $B^{2N}(R)\cong B\subset M$ where $B^{2N}(R)$ is the open ball of radius $R$ in $\R^{2N}$.
For an integer $k\geq1$ and a positive number $\delta>0$ we choose a smooth function $\rho:[0,R^2]\to\R$ satisfying the following
\beq
\left\{
\begin{aligned}
\label{eqn:properties_rho}
&\frac{\pi}{2k}\geq\rho'(r)>0\;,\\
&\rho'(r)=\frac{\pi}{2k}\quad\Longleftrightarrow \quad r=\tfrac12R^2\;,\\
&\rho'|_{\big[\frac89R^2,R^2\big]}=\delta>0\;.\\
 \end{aligned}\right.
\eeq
We set for $1\leq\nu\leq N$
\beq\label{eqn:zeta}
\zeta(\nu):=
\begin{cases}
1 & \nu=N\\
\frac{9}{10} & \text{ else} 
\end{cases}
\eeq
and define
\bea
H:B^{2N}(R)&\to\R\\
z&\mapsto \rho\left(\sum_{\nu=1}^N\zeta(\nu)|z_\nu|^2\right)\;.
\eea
We denote by $\phi_H^t:B^{2N}(R)\to B^{2N}(R)$ the induced Hamiltonian flow. We recall that the Hamiltonian flow of $z\mapsto |z|^2$ is given by $z\mapsto \exp(2it)z$ thus
\beq\label{eqn:flow_of_H}
\big(\phi_H^t(z)\big)_\nu=\exp\left[\rho'\left(\sum_{\nu=1}^N\zeta(\nu)|z_\nu|^2\right)2i\zeta(\nu)t\right]z_\nu\;.
\eeq
We point out that $\phi_H^t$ preserves the quantities $|z_\nu|$, $\nu=1,\ldots,N$. 
\begin{Lemma}\label{lem:fix_points_of_phi_H_2k}
The fixed points of $\phi_H^{2k}$ are precisely $z=0$ and the circle
\beq
C:=\Big\{(z_1,\ldots,z_N)\in B^{2N}(R)\mid |z_N|^2=\tfrac12R^2\text{ and }z_1=\ldots=z_{N-1}=0\Big\}\;.
\eeq 
Moreover, $\phi_H$ acts on $C$ by rotation of the last coordinate by an angle of $\frac{\pi}{k}$.
 \end{Lemma}
\begin{proof}
Assume $\phi_H^{2k}(z)=z$ which is equivalent to 
\beq
\exp\left[\rho'\left(\sum_{\nu=1}^N\zeta(\nu)|z_\nu|^2\right)2i\zeta(\nu)2k\right]z_\nu=z_\nu,\quad\nu=1,\ldots,N\;,
\eeq
thus, either $z_\nu=0$ or
\beq
\rho'\left(\sum_{\nu=1}^N\zeta(\nu)|z_\nu|^2\right)4k\zeta(\nu)\in2\pi\Z\;.
\eeq
From $\rho'(r)\leq\frac{\pi}{2k}$ we conclude that $z_1=\ldots=z_{N-1}=0$. Moreover, $z_N=0$ or 
\beq
\rho'\left(\sum_{\nu=1}^N\zeta(\nu)|z_\nu|^2\right)=\rho'(|z_N|^2)=\frac{\pi}{2k}
\eeq
holds. In summary, either $z=0$ or $z\in C$. This together with \eqref{eqn:flow_of_H} proves the Lemma.
\end{proof}
We now perturb $H$. For this we fix a smooth cut-off function $\beta:[0,R^2]\to[0,1]$ satisfying
\beq
\beta|_{\big[\frac13R^2,\frac23R^2\big]}=1\quad\text{and}\quad\beta|_{\big[0,\frac19R^2\big]\cup\big[\frac89R^2,R^2\big]}=0
\eeq
and set
\beq
F(z):=\beta(|z_N|^2)\cdot\text{Re}\left(\frac{z_N^k}{|z_N|^k}\right):B^{2N}(R)\to\R
\eeq
where $\text{Re}$ is the real part. If we introduce new coordinates $(z_1,\ldots,z_{N-1},r,\vartheta)$, where $z_N=r\exp(i\vartheta)$, the function $F$ equals
\beq
F(z)=\beta(r^2)\cos\big(k\vartheta\big)\;.
\eeq
We point out  that the Hamiltonian diffeomorphism $\phi_H\circ\phi_{\epsilon F}$ maps $B^{2N}(R)$ into itself.
\begin{Lemma}\label{lem:small_epsilon}
There exists $\epsilon_0>0$ such that for all $0<\epsilon<\epsilon_0$
\beq
\#\mathscr{C}^{2k}(\phi_H\circ\phi_{\epsilon F})=1\;.
\eeq
%
%
\end{Lemma}

\begin{proof}
We set
\beq
D:=\left\{(z_1,\ldots,z_{N-1},r,\vartheta)\in C\mid \vartheta=\frac{j\pi}{k},\,j=0,\ldots,2k-1\right\}
\eeq
where $C$ is defined in Lemma \ref{lem:fix_points_of_phi_H_2k}. The same lemma implies that $\phi_H$ acts on $D$ as a cyclic permutation sending $\frac{j\pi}{k}$ to $\frac{(j+1)\pi}{k}$. Moreover, we have
\beq
\phi_{\epsilon F}z=z
\eeq
for $z\in D$ since $D\subset\Crit F$. In particular, $D$ corresponds precisely to a single element in $\mathscr{C}^{2k}(\phi_H\circ\phi_{\epsilon F})$. It remains to show that there are no other $2k$-cycles. We prove something stronger, namely that for sufficiently small $\epsilon>0$ the only other fixed point of $\big(\phi_H\circ\phi_{\epsilon F}\big)^{2k}$ is $z=0$.

For $0<a<b$ we set
\beq
A(a,b):=\big\{(z_1,\ldots,z_{N-1},r,\vartheta)\in B^{2N}(R)\mid r\in[aR^2,bR^2]\big\}\;.
\eeq
We observe that on $A(\frac13,\frac23)$ we have $\beta=1$ and thus the flow of $\epsilon F$ is given by 
\beq\label{eqn:flow_of_F}
(z_1,\ldots,z_{N-1},r,\vartheta)\mapsto (z_1,\ldots,z_{N-1},\sqrt{-2\epsilon k\sin(k\vartheta)t+r^2},\vartheta)\;.
\eeq
In particular, if we set
\beq
\bar\epsilon:=\frac{7R^4}{324k^2}
\eeq
then for $0<\epsilon<\bar\epsilon$ we conclude that
\beq
\big(\phi_H\circ\phi_{\epsilon F}\big)^{2k}\left(A(\tfrac49,\tfrac59)\right)\subset A(\tfrac13,\tfrac23)\;,
\eeq
since $\phi_H^t$ preserves the $r$ coordinate. Fix $w\in A(\tfrac49,\tfrac59)$ with $\big(\phi_H\circ\phi_{\epsilon F}\big)^{2k}(w)=w$ and set for $j=0,\ldots,2k$
\bea
z^j_{\nu}&:=P_{z_\nu}\Big(\big(\phi_H\circ\phi_{\epsilon F}\big)^{j}(w)\Big)\;,\quad \nu=1,\ldots,N-1,\\
r^j&:=P_r\Big(\big(\phi_H\circ\phi_{\epsilon F}\big)^{j}(w)\Big)\;,\\
\vartheta^j&:=P_\vartheta\Big(\big(\phi_H\circ\phi_{\epsilon F}\big)^{j}(w)\Big)\;,
\eea
where $P_{z_\nu}$, $P_{r}$, and $P_{\vartheta}$ are the projections on the respective coordinates. It follows from equation $\eqref{eqn:flow_of_F}$ that
\beq
P_{z_\nu}\Big(\big(\phi_H\circ\phi_{\epsilon F}\big)^{j}(w)\Big)=P_{z_\nu}\big(\phi_H^j(w)\big)\quad \nu=1,\ldots,N-1\;.
\eeq
By the same argument as in the proof of Lemma \ref{lem:fix_points_of_phi_H_2k} we conclude
\beq
z^j_\nu=0\quad\forall \nu=1,\ldots,N-1\text{ and }\forall j=0,\ldots,2k\;.
\eeq
Next, it follows from the flow equations \eqref{eqn:flow_of_H} and $\eqref{eqn:flow_of_F}$ 
\beq
0< \vartheta_{j+1}-\vartheta_j\leq\frac{\pi}{k} \mod 2\pi\;.
\eeq
By \eqref{eqn:properties_rho} equality holds if and only if $r_{j+1}=\frac12R^2$. Using again $\big(\phi_H\circ\phi_{\epsilon F}\big)^{2k}(w)=w$ we deduce
\beq
\vartheta_{2k}-\vartheta_0=0\mod 2\pi
\eeq
and therefore
\beq
r_0=r_1=\ldots=r_{2k}=\tfrac12R^2\;.
\eeq
In summary
\beq
w=(0,\ldots,0,\tfrac12R^2,\vartheta_0)
\eeq
with $\vartheta_0\in\frac{\pi}{k}\Z$, i.e.~$w\in D$. Thus, we proved that the only $2k$-cycle of $\phi_H\circ\phi_{\epsilon F}$ in the region $A(\tfrac49,\tfrac59)$ is the one corresponding to the set $D$. Therefore it remains to prove that after possibly shrinking $\bar\epsilon$ there are no other fixed points of $\big(\phi_H\circ\phi_{\epsilon F}\big)^{2k}$ outside $A(\tfrac49,\tfrac59)$ except for $z=0$. We argue by contradiction. 

We assume that there exists a sequence $\epsilon_m\to0$ and a sequence $(z^m)_{m\in\N}$ of points in $B^{2N}(R)\setminus A(\tfrac49,\tfrac59)$ with
\beq
\big(\phi_H\circ\phi_{\epsilon_m F}\big)^{2k}(z^m)=z^m\quad\forall m\in\N\;.
\eeq
By compactness we may assume that $z^m\to z^*\in B^{2N}(R)\setminus \text{int}A(\tfrac49,\tfrac59)$ with
\beq
\phi_H^{2k}(z^*)=z^*\;.
\eeq
It follows from Lemma \ref{lem:fix_points_of_phi_H_2k} that $z^*=0$ and thus for $M$ sufficiently large
\beq
z^m\in B^{2N}(\tfrac13R)\quad\forall m\geq M\;.
\eeq
Then by definition of $\beta$ the restriction of $\phi_{\epsilon_m F}$ to the ball $B^{2N}(\frac13R)$ equals the identity. Moreover, since $\phi_H$ fixes all balls centered at zero we have
\beq
z^m=\big(\phi_H\circ\phi_{\epsilon_m F}\big)^{2k}(z^m)=\phi_H^{2k}(z^m)\quad\forall m\geq M\;.
\eeq
Applying again Lemma \ref{lem:fix_points_of_phi_H_2k} we conclude that $z^m=0$ for all $m\geq M$. This proves the Lemma.
\end{proof}

\begin{Rmk}
Proposition \ref{prop:Milnor} together with Lemma \ref{lem:small_epsilon} implies that for all $0<\epsilon<\epsilon_0$ the Hamiltonian diffeomorphism $\phi_H\circ\phi_{\epsilon F}:B^{2N}(R)\to B^{2N}(R)$ has no square root. 
\end{Rmk}

We are now in the position to prove Theorem \ref{thm:main}.

\begin{proof}[Proof of Theorem \ref{thm:main}] 
We choose $k\in\Z$, $\delta>0$ and $0<\epsilon<\epsilon_0$ (cp.~Lemma \ref{lem:small_epsilon}) so that the Hamiltonian diffeomorphism 
\beq
\phi_H\circ\phi_{\epsilon F}:B^{2N}(R)\to B^{2N}(R)
\eeq 
has precisely one $2k$-cycle. By construction $\phi_H\circ\phi_{\epsilon F}$ equals the map
\beq
(z_1,\ldots,z_N)\mapsto\left(e^{\frac{9i\delta}{5}}z_1,\ldots, e^{\frac{9i\delta}{5}}z_{N-1},e^{2i\delta}z_N\right)
\eeq
near the boundary of $B^{2N}(R)$. Indeed, if $z\in\partial B^{2N}(R)$ then we conclude
\beq
\sum_{\nu=1}^N\zeta(\nu)|z_\nu|^2\geq \frac{9}{10}\sum_{\nu=1}^N|z_\nu|^2=\frac{9}{10}R^2>\frac89R^2
\eeq
and therefore $\rho'(\sum_{\nu=1}^N\zeta(\nu)|z_\nu|^2)=\delta$. Next, we extend the Hamiltonian function of $\phi_H\circ\phi_{\epsilon F}$ to $\widetilde{H}:S^1\times M\to\R$ which we can choose to be autonomous outside the Darboux ball $B$. If we choose $\delta>0$ sufficiently small we can guarantee that outside $B$ the only periodic orbits of $\widetilde{H}$ of period less or equal to $2k$ are critical points of $\widetilde{H}$, see \cite{Hofer_Zehner_Book}, in particular line 4 \& 5  on page 185. In particular, $\phi_{\widetilde{H}}$ has still precisely one $2k$-cycle. Finally, by choosing $k$ sufficiently large and $\delta$ and $\epsilon$ sufficiently small,  $\phi_H\circ\phi_{\epsilon F}$ and thus $\phi_{\widetilde{H}}$ can be chosen to lie in an arbitrary $C^\infty$-neighborhood of the identity on $B^{2N}(R)$ resp.~$M$. Therefore, with Proposition \ref{prop:Milnor} the Theorem follows.
\end{proof}

%
\bibliographystyle{amsalpha}
\bibliography{../../../../../Bibtex/bibtex_paper_list}
\end{document}